\documentclass{amsart}
\usepackage{amsfonts}

\newtheorem{theorem}{Theorem}
\theoremstyle{plain}

\newtheorem{definition}{Definition}

\newtheorem{lemma}{Lemma}

\newtheorem{proposition}{Proposition}
\newtheorem{remark}{Remark}

\numberwithin{equation}{section}
\input{tcilatex}

\begin{document}
\title[Inequalities for $s$-convex functions and applications]{Integral
inequalities for mappings whose derivatives are $s$-convex in the second
sense and applications to special means for positive real numbers}
\author{Mevl\"{u}t TUN\c{C}$^{\clubsuit }$}
\address{$^{\clubsuit ,\spadesuit }$Mustafa Kemal University, Faculty of
Science and Arts, Department of Mathematics, 31000, Hatay, Turkey}
\email{mevluttttunc@gmail.com}
\urladdr{}
\thanks{}
\author{Sevil BALGE\c{C}T\.{I}$^{\spadesuit }$}
\email{sevilbalgecti@gmail.com }
\urladdr{}
\thanks{$^{\clubsuit }Corresponding$ $Author$}
\dedicatory{}
\thanks{}

\begin{abstract}
In this paper, the authors establish a new type integral inequalities for
differentiable $s$-convex functions in the second sense. By the well-known H%
\"{o}lder inequality and power mean inequality, they obtain some integral
inequalities related to the $s$-convex functions and apply these
inequalities to special means for positive real numbers.
\end{abstract}

\maketitle

\section{Introduction}

\subsection{Definitions}

\begin{definition}
\cite{mit2} A function $\varphi :I\rightarrow 
\mathbb{R}
$ is said to be convex on $I$ if inequality%
\begin{equation}
\varphi \left( tx+\left( 1-t\right) y\right) \leq t\varphi \left( x\right)
+\left( 1-t\right) \varphi \left( y\right)  \label{101}
\end{equation}%
holds for all $x,y\in I$ and $t\in \left[ 0,1\right] $. We say that $\varphi 
$ is concave if $(-\varphi )$ is convex.
\end{definition}

\begin{definition}
\cite{hud}\textit{\ Let }$s\in \left( 0,1\right] .$\textit{\ A function }$%
\varphi :\left( 0,\infty \right] \rightarrow \left[ 0,\infty \right] $%
\textit{\ is said to be }$s$-\textit{convex in the second sense if \ \ \ \ \
\ \ \ \ \ \ \ }%
\begin{equation}
\varphi \left( tx+\left( 1-t\right) y\right) \leq t^{s}\varphi \left(
x\right) +\left( 1-t\right) ^{s}\varphi \left( y\right) ,  \label{105}
\end{equation}%
\textit{for all }$x,y\in \left( 0,b\right] $\textit{\ \ and }$t\in \left[ 0,1%
\right] $\textit{. This class of }$s$\textit{-convex functions is usually
denoted by }$K_{s}^{2}$\textit{.}
\end{definition}

Certainly, $s$-convexity means just ordinary convexity when $s=1$.

\subsection{Theorems}

\begin{theorem}
\textbf{The Hermite-Hadamard inequality:} Let $\varphi :I\subseteq 
\mathbb{R}
\rightarrow 
\mathbb{R}
$ be a convex function and $u,v\in I$ with $u<v$. The following double
inequality:%
\begin{equation}
\varphi \left( \frac{u+v}{2}\right) \leq \frac{1}{v-u}\int_{u}^{v}\varphi
\left( x\right) dx\leq \frac{\varphi \left( u\right) +\varphi \left(
v\right) }{2}  \label{110}
\end{equation}%
is known in the literature as Hadamard's inequality (or Hermite-Hadamard
inequality) for convex functions. If $\varphi $ is a positive concave
function, then the inequality is reversed.
\end{theorem}

\begin{theorem}
\cite{ssd6} \textit{Suppose that }$\varphi :\left[ 0,\infty \right)
\rightarrow \left[ 0,\infty \right) $\textit{\ is an }$s-$\textit{convex
function in the second sense, where }$s\in \left( 0,1\right] $\textit{, and
let }$a,b\in \left[ 0,\infty \right) ,$ $a<b.$\textit{\ If }$\varphi \in
L_{1}\left( \left[ 0,1\right] \right) $\textit{, then the following
inequalities hold:}%
\begin{equation}
2^{s-1}\varphi \left( \frac{u+v}{2}\right) \leq \frac{1}{v-u}%
\int_{u}^{v}\varphi \left( x\right) dx\leq \frac{\varphi \left( u\right)
+\varphi \left( v\right) }{s+1}.  \label{109}
\end{equation}%
The constant $k=\frac{1}{s+1}$ is the best possible in the second inequality
in (\ref{109}). The above inequalities are sharp. If $\varphi $ is an $s$%
-concave function in the second sense, then the inequality is reversed.
\end{theorem}

\begin{theorem}
Let $\varphi :I\subseteq 
\mathbb{R}
\rightarrow 
\mathbb{R}
$ be a convex function on the interval $I$ of real numbers and $a,b\in I$
with $a<b$. The inequality%
\begin{equation*}
\frac{1}{v-u}\int_{u}^{v}\varphi \left( x\right) dx\leq \frac{1}{2}\left[
\varphi \left( \frac{u+v}{2}\right) +\frac{\varphi \left( u\right) +\varphi
\left( v\right) }{2}\right]
\end{equation*}%
is known as \textbf{Bullen's inequality} for convex functions \cite[p.39]%
{dr2}.
\end{theorem}

In \cite{dra2}, Dragomir and Agarwal\ obtained inequalities for
differentiable convex mappings which are connected to Hadamard's inequality,
as follow:

\begin{theorem}
Let $f:I\subseteq 
\mathbb{R}
\rightarrow 
\mathbb{R}
$ be a differentiable mapping on $I^{\circ }$, where $a,b\in I$, with $a<b$.
If $\left\vert f^{\prime }\right\vert ^{q}$ is convex on $\left[ a,b\right] $%
, then the following inequality holds: 
\begin{equation}
\left\vert \frac{f\left( a\right) +f\left( b\right) }{2}-\frac{1}{b-a}%
\int_{a}^{b}f\left( x\right) dx\right\vert \leq \frac{b-a}{8}\left[
\left\vert f^{\prime }\left( a\right) \right\vert +\left\vert f^{\prime
}\left( b\right) \right\vert \right] .  \label{11}
\end{equation}
\end{theorem}

In \cite{cem}, Pearce and Pe\v{c}ari\'{c} obtained inequalities for
differentiable convex mappings which are connected with Hadamard's
inequality, as follow:

\begin{theorem}
Let $f:I\subseteq 
\mathbb{R}
\rightarrow 
\mathbb{R}
$ be differentiable mapping on $I^{\circ }$, where $a,b\in I$, with $a<b$.
If $\left\vert f^{\prime }\right\vert ^{q}$ is convex on $\left[ a,b\right] $
for some $q\geq 1$, then the following inequality holds:%
\begin{equation}
\left\vert \frac{f\left( a\right) +f\left( b\right) }{2}-\frac{1}{b-a}%
\int_{a}^{b}f\left( x\right) dx\right\vert \leq \frac{b-a}{4}\left( \frac{%
\left( \left\vert f^{\prime }\left( a\right) \right\vert ^{q}+\left\vert
f^{\prime }\left( b\right) \right\vert ^{q}\right) }{2}\right) ^{\frac{1}{q}%
}.  \label{12}
\end{equation}%
If $\left\vert f^{\prime }\right\vert ^{q}$ is concave on $\left[ a,b\right] 
$ for some $q\geq 1$, then 
\begin{equation}
\left\vert \frac{f\left( a\right) +f\left( b\right) }{2}-\frac{1}{b-a}%
\int_{a}^{b}f\left( x\right) dx\right\vert \leq \frac{b-a}{4}\left\vert
f^{\prime }\left( \frac{a+b}{2}\right) \right\vert .  \label{13}
\end{equation}
\end{theorem}

In \cite{alo2}, Alomari, Darus and K\i rmac\i\ obtained inequalities for
differentiable $s$-convex and concave mappings which are connected with
Hadamard's inequality, as follow:

\begin{theorem}
Let $f:I\subseteq \left[ 0,\infty \right) \rightarrow 
\mathbb{R}
$ be differentiable mapping on $I^{\circ }$ such that $f^{\prime }\in L\left[
a,b\right] ,$\ where $a,b\in I$ with $a<b$. If $\left\vert f^{\prime
}\right\vert ^{q}$, $q\geq 1$ is concave on $\left[ a,b\right] $ for some
fixed $s\in \left( 0,1\right] $, then the following inequality holds:%
\begin{eqnarray}
&&\left\vert \frac{f\left( a\right) +f\left( b\right) }{2}-\frac{1}{b-a}%
\int_{a}^{b}f\left( x\right) dx\right\vert  \label{14} \\
&\leq &\left( \frac{b-a}{4}\right) \left( \frac{q-1}{2q-1}\right) ^{1-\frac{1%
}{q}}\left[ \left\vert f^{\prime }\left( \frac{3a+b}{4}\right) \right\vert
+\left\vert f^{\prime }\left( \frac{a+3b}{4}\right) \right\vert \right] . 
\notag
\end{eqnarray}
\end{theorem}

In \cite{ms}, Tun\c{c} and Balge\c{c}ti\ obtained inequalities for
differentiable convex functions which are connected with a new type integral
inequality, as follow:

\begin{lemma}
\label{l1}Let $f:J\rightarrow 
\mathbb{R}
$ be a differentiable function on $J^{\circ }$. If $f^{\prime }\in L\left[
a,b\right] $, then
\end{lemma}

\begin{eqnarray}
&&\frac{1}{b-a}\int_{a}^{b}f\left( x\right) dx-\frac{1}{2}\left( \frac{%
bf\left( a\right) -af\left( b\right) }{b-a}+f\left( \frac{a+b}{2}\right)
\right)  \label{se1} \\
&=&\frac{1}{4}\int_{0}^{1}\left( tb+\left( 1-t\right) a\right) f^{\prime
}\left( \frac{1-t}{2}b+\frac{1+t}{2}a\right) dt  \notag \\
&&+\frac{1}{4}\int_{0}^{1}\left( ta+\left( 1-t\right) b\right) f^{\prime
}\left( \frac{1-t}{2}a+\frac{1+t}{2}b\right) dt  \notag
\end{eqnarray}

for each $t\in \left[ 0,1\right] $ and $a,b\in J.$

\begin{theorem}
\cite{ms}Let $f:J\rightarrow 
\mathbb{R}
$ be a differentiable function on $J^{\circ }$. If $\left\vert f^{\prime
}\right\vert $\ is convex on $J$, then%
\begin{eqnarray}
&&\left\vert \frac{1}{b-a}\int_{a}^{b}f\left( x\right) dx-\frac{1}{2}\left( 
\frac{bf\left( a\right) -af\left( b\right) }{b-a}+f\left( \frac{a+b}{2}%
\right) \right) \right\vert  \label{s1} \\
&\leq &\left( \frac{5}{48}a+\frac{7}{48}b\right) \left\vert f^{\prime
}\left( a\right) \right\vert +\left( \frac{7}{48}a+\frac{5}{48}b\right)
\left\vert f^{\prime }\left( b\right) \right\vert  \notag
\end{eqnarray}%
for each $a,b\in J$.
\end{theorem}

\begin{theorem}
\cite{ms}Let$\ f:J\rightarrow 
\mathbb{R}
$ be$\ $a\ differentiable\ function on$\ J^{\circ }$. If\ $\left\vert
f^{\prime }\right\vert ^{q}\ $is\ convex\ on$\ \left[ a,b\right] \ $and$\
q>1\ $with$\ \frac{1}{p}+\frac{1}{q}=1$,$\ $then%
\begin{eqnarray}
&&\left\vert \frac{1}{b-a}\int_{a}^{b}f\left( x\right) dx-\frac{1}{2}\left( 
\frac{bf\left( a\right) -af\left( b\right) }{b-a}+f\left( \frac{a+b}{2}%
\right) \right) \right\vert  \label{s2} \\
&\leq &\frac{1}{4^{1+1/q}}L_{p}\left( a,b\right) \left[ \left[ \left\vert
f^{\prime }\left( b\right) \right\vert ^{q}+3\left\vert f^{\prime }\left(
a\right) \right\vert ^{q}\right] ^{\frac{1}{q}}+\left[ \left\vert f^{\prime
}\left( a\right) \right\vert ^{q}+3\left\vert f^{\prime }\left( b\right)
\right\vert ^{q}\right] ^{\frac{1}{q}}\right]  \notag
\end{eqnarray}
\end{theorem}

\begin{theorem}
\cite{ms}Let\ $f:J\rightarrow 
\mathbb{R}
$\ be\ a\ differentiable\ function on\ $J^{\circ }$. If$\ $\ $\left\vert
f^{\prime }\right\vert ^{q}\ $is$\ $convex\ on$\ \left[ a,b\right] \ $and$\
q\geq 1$, then%
\begin{eqnarray}
&&\left\vert \frac{1}{b-a}\int_{a}^{b}f\left( x\right) dx-\frac{1}{2}\left( 
\frac{bf\left( a\right) -af\left( b\right) }{b-a}+f\left( \frac{a+b}{2}%
\right) \right) \right\vert  \label{s3} \\
&\leq &\frac{A^{1-\frac{1}{q}}\left( a,b\right) }{4\times 12^{\frac{1}{q}}}%
\left\{ \left[ \left\vert f^{\prime }(b)\right\vert ^{q}\left( 2a+b\right)
+\left\vert f^{\prime }(a)\right\vert ^{q}\left( 4a+5b\right) \right] ^{%
\frac{1}{q}}\right.  \notag \\
&&+\left. \left[ \left\vert f^{\prime }(a)\right\vert ^{q}\left( a+2b\right)
+\left\vert f^{\prime }(b)\right\vert ^{q}\left( 5a+4b\right) \right] ^{%
\frac{1}{q}}\right\}  \notag
\end{eqnarray}
\end{theorem}

For recent results and generalizations concerning Hadamard's inequality and
concepts of convexity and $s$-convexity see \cite{alo2}-\cite{ms} and the
references therein.

Throughout this paper we will use the following notations and conventions.
Let $J=\left[ 0,\infty \right) \subset 
\mathbb{R}
=\left( -\infty ,+\infty \right) ,$ and $u,v\in J$ with $0<u<v$ and $%
f^{\prime }\in L\left[ u,v\right] $ and%
\begin{equation*}
A\left( u,v\right) =\frac{u+v}{2},\text{ }G\left( u,v\right) =\sqrt{uv},
\end{equation*}
\begin{equation*}
L_{p}\left( u,v\right) =\left( \frac{v^{p+1}-u^{p+1}}{\left( p+1\right)
\left( v-u\right) }\right) ^{1/p},\text{ }u\neq v,\text{ }p\in 
\mathbb{R}
,\text{ }p\neq -1,0
\end{equation*}%
be the arithmetic mean, geometric mean, generalized logarithmic mean for $%
u,v>0$ respectively.

\section{Inequalities for $s$-convex functions and applications}

\begin{theorem}
\label{2}Let $f:J\rightarrow 
\mathbb{R}
$ be a differentiable function on $J^{\circ }$. If $\left\vert f^{\prime
}\right\vert $\ is $s$-convex on $\left[ a,b\right] $ for some fixed $s\in
\left( 0,1\right] $, then%
\begin{eqnarray}
&&\left\vert \frac{1}{b-a}\int_{a}^{b}f\left( x\right) dx-\frac{1}{2}\left( 
\frac{bf\left( a\right) -af\left( b\right) }{b-a}+f\left( \frac{a+b}{2}%
\right) \right) \right\vert  \label{se2} \\
&\leq &\frac{b\left( s2^{s+1}+s+2\right) +a\left( 2^{s+2}-s-2\right) }{%
2^{s+2}\left( s+1\right) \left( s+2\right) }\left\vert f^{\prime }\left(
a\right) \right\vert  \notag \\
&&+\frac{a\left( s2^{s+1}+s+2\right) +b\left( 2^{s+2}-s-2\right) }{%
2^{s+2}\left( s+1\right) \left( s+2\right) }\left\vert f^{\prime }\left(
b\right) \right\vert  \notag
\end{eqnarray}%
for each $x\in \left[ a,b\right] $.
\end{theorem}

\begin{proof}
Using Lemma \ref{l1} and from properties of modulus, and since $\left\vert
f^{\prime }\right\vert $\ is $s$-convex on $J$, then we obtain%
\begin{eqnarray*}
&&\left\vert \frac{1}{b-a}\int_{a}^{b}f\left( x\right) dx+\frac{af\left(
b\right) -bf\left( a\right) }{2\left( b-a\right) }-\frac{1}{2}f\left( \frac{%
a+b}{2}\right) \right\vert \\
&\leq &\frac{1}{4}\int_{0}^{1}\left( tb+\left( 1-t\right) a\right)
\left\vert f^{\prime }\left( \left( \frac{1-t}{2}\right) ^{s}b+\left( \frac{%
1+t}{2}\right) ^{s}a\right) \right\vert dt \\
&&+\frac{1}{4}\int_{0}^{1}\left( ta+\left( 1-t\right) b\right) \left\vert
f^{\prime }\left( \left( \frac{1-t}{2}\right) ^{s}a+\left( \frac{1+t}{2}%
\right) ^{s}b\right) \right\vert dt \\
&\leq &\frac{1}{4}\int_{0}^{1}\left( tb+\left( 1-t\right) a\right) \left[
\left( \frac{1-t}{2}\right) ^{s}\left\vert f^{\prime }\left( b\right)
\right\vert +\left( \frac{1+t}{2}\right) ^{s}\left\vert f^{\prime }\left(
a\right) \right\vert \right] dt \\
&&+\frac{1}{4}\int_{0}^{1}\left( ta+\left( 1-t\right) b\right) \left[ \left( 
\frac{1-t}{2}\right) ^{s}\left\vert f^{\prime }\left( a\right) \right\vert
+\left( \frac{1+t}{2}\right) ^{s}\left\vert f^{\prime }\left( b\right)
\right\vert \right] dt \\
&\leq &\frac{1}{4}\int_{0}^{1}\left( tb+\left( 1-t\right) a\right) \left( 
\frac{1-t}{2}\right) ^{s}\left\vert f^{\prime }\left( b\right) \right\vert
dt+\frac{1}{4}\int_{0}^{1}\left( tb+\left( 1-t\right) a\right) \left( \frac{%
1+t}{2}\right) ^{s}\left\vert f^{\prime }\left( a\right) \right\vert dt \\
&&+\frac{1}{4}\int_{0}^{1}\left( ta+\left( 1-t\right) b\right) \left( \frac{%
1-t}{2}\right) ^{s}\left\vert f^{\prime }\left( a\right) \right\vert dt+%
\frac{1}{4}\int_{0}^{1}\left( ta+\left( 1-t\right) b\right) \left( \frac{1+t%
}{2}\right) ^{s}\left\vert f^{\prime }\left( b\right) \right\vert dt \\
&=&\frac{1}{2^{s+2}}\allowbreak \frac{as+a+b}{\left( s+1\right) \left(
s+2\right) }\left\vert f^{\prime }\left( b\right) \right\vert +\frac{1}{%
2^{s+2}}\frac{b\left( s2^{s+1}+1\right) +a\left( 2^{s+2}-s-3\right) }{\left(
s+1\right) \left( s+2\right) }\left\vert f^{\prime }\left( a\right)
\right\vert \\
&&+\frac{1}{2^{s+2}}\frac{bs+b+a}{\left( s+1\right) \left( s+2\right) }%
\left\vert f^{\prime }\left( a\right) \right\vert +\frac{1}{2^{s+2}}\frac{%
a\left( s2^{s+1}+1\right) +b\left( 2^{s+2}-s-3\right) }{\left( s+1\right)
\left( s+2\right) }\left\vert f^{\prime }\left( b\right) \right\vert
\end{eqnarray*}%
\begin{eqnarray*}
&=&\frac{1}{2^{s+2}}\left( \frac{as+a+b}{\left( s+1\right) \left( s+2\right) 
}+\frac{a\left( s2^{s+1}+1\right) +b\left( 2^{s+2}-s-3\right) }{\left(
s+1\right) \left( s+2\right) }\right) \allowbreak \left\vert f^{\prime
}\left( b\right) \right\vert \\
&&+\frac{1}{2^{s+2}}\left( \frac{bs+b+a}{\left( s+1\right) \left( s+2\right) 
}+\frac{b\left( s2^{s+1}+1\right) +a\left( 2^{s+2}-s-3\right) }{\left(
s+1\right) \left( s+2\right) }\right) \left\vert f^{\prime }\left( a\right)
\right\vert \\
&=&\frac{1}{2^{s+2}}\left( \frac{a\left( s2^{s+1}+s+2\right) +b\left(
2^{s+2}-s-2\right) }{\left( s+1\right) \left( s+2\right) }\right) \left\vert
f^{\prime }\left( b\right) \right\vert \\
&&+\frac{1}{2^{s+2}}\left( \frac{b\left( s2^{s+1}+s+2\right) +a\left(
2^{s+2}-s-2\right) }{\left( s+1\right) \left( s+2\right) }\right) \left\vert
f^{\prime }\left( a\right) \right\vert
\end{eqnarray*}
\end{proof}

\begin{proposition}
\label{21} Let $a,b\in J^{\circ },\ 0<a<b$ and $s\in \left( 0,1\right] $ then%
\begin{eqnarray}
&&\left\vert L_{s}^{s}\left( a,b\right) +\frac{\left( s-1\right) G^{2}\left(
a,b\right) L_{s-2}^{s-2}\left( a,b\right) -A^{s}\left( a,b\right) }{2}%
\right\vert  \notag \\
&\leq &\frac{s\left[ \left( ab^{s-1}+a^{s-1}b\right) \left(
s2^{s+1}+s+2\right) +\left( a^{s}+b^{s}\right) \left( 2^{s+2}-s-2\right) %
\right] }{2^{s+2}\left( s+1\right) \left( s+2\right) }  \label{se3}
\end{eqnarray}
\end{proposition}

\begin{proof}
The proof follows from (\ref{se2}) applied\ to the $s$-convex function $%
f\left( x\right) =x^{s}$ and $\left\vert f^{\prime }\left( x\right)
\right\vert =sx^{s-1}.$
\end{proof}

\begin{proposition}
\label{22}Let $a,b\in J^{\circ }$,$\ 0<a<b$, $s\in \left( 0,1\right) $ then%
\begin{eqnarray}
&&\left\vert \frac{L_{s}^{s}\left( a,b\right) }{1-s}-\frac{sG^{2}\left(
a,b\right) L_{-1-s}^{-1-s}\left( a,b\right) +A^{1-s}\left( a,b\right) }{%
2\left( 1-s\right) }\ \right\vert  \label{se4} \\
&\leq &\frac{1}{2^{s+2}b^{s}}\left( \frac{a\left( s2^{s+1}+s+2\right)
+b\left( 2^{s+2}-s-2\right) }{\left( s+1\right) \left( s+2\right) }\right) 
\notag \\
&&+\frac{1}{2^{s+2}a^{s}}\left( \frac{a\left( s-2^{s+2}+2\right) -b\left(
s2^{s+1}+s+2\right) }{\left( s+1\right) \left( s+2\right) }\right)
\allowbreak  \notag
\end{eqnarray}
\end{proposition}

\begin{proof}
The proof follows from (\ref{se2}) applied\ to the $s$-convex function $%
f\left( x\right) =\frac{x^{1-s}}{1-s}$ and $\left\vert f^{\prime
}(x)\right\vert $ $=1/x^{s}$ with $s\in \left( 0,1\right) $.
\end{proof}

\begin{remark}
In (\ref{se2}), $\allowbreak $(\ref{se3}), if we take $s\rightarrow 1$, then
(\ref{se2}), $\allowbreak $(\ref{se3}) reduces to (\ref{s1}), \cite[%
Proposition 2]{ms}, respectively.
\end{remark}

\begin{theorem}
\label{3}Let $f:J\rightarrow 
\mathbb{R}
$ be a differentiable function on $J^{\circ }.$ If $\left\vert f^{\prime
}\right\vert $\ is $s$-convex on $\left[ a,b\right] $ for some fixed $s\in
\left( 0,1\right] $ and $q>1$ with $\frac{1}{p}+\frac{1}{q}=1$, then%
\begin{eqnarray}
&&\left\vert \frac{1}{b-a}\int_{a}^{b}f\left( x\right) dx-\frac{1}{2}\left( 
\frac{bf\left( a\right) -af\left( b\right) }{b-a}+f\left( \frac{a+b}{2}%
\right) \right) \right\vert  \label{se5} \\
&\leq &\frac{L_{p}\left( a,b\right) }{4\left( 2^{s}\left( s+1\right) \right)
^{\frac{1}{q}}}\left\{ \left( \left\vert f^{\prime }\left( b\right)
\right\vert ^{q}+\left( 2^{s+1}-1\right) \left\vert f^{\prime }\left(
a\right) \right\vert ^{q}\right) ^{\frac{1}{q}}\right.  \notag \\
&&+\left. \left( \left\vert f^{\prime }\left( a\right) \right\vert
^{q}+\left( 2^{s+1}-1\right) \left\vert f^{\prime }\left( b\right)
\right\vert ^{q}\right) ^{\frac{1}{q}}\right\}  \notag
\end{eqnarray}%
for each $x\in \left[ a,b\right] $.
\end{theorem}

\begin{proof}
Using Lemma \ref{l1} and from properties of modulus, and since $\left\vert
f^{\prime }\right\vert $\ is $s$-convex on $J$, then we obtain%
\begin{eqnarray}
&&  \label{3.1} \\
&&\left\vert \frac{1}{b-a}\int_{a}^{b}f\left( x\right) dx+\frac{af\left(
b\right) -bf\left( a\right) }{2\left( b-a\right) }-\frac{1}{2}f\left( \frac{%
a+b}{2}\right) \right\vert  \notag \\
&\leq &\frac{1}{4}\int_{0}^{1}\left( tb+\left( 1-t\right) a\right)
\left\vert f^{\prime }\left( \left( \frac{1-t}{2}\right) ^{s}b+\left( \frac{%
1+t}{2}\right) ^{s}a\right) \right\vert dt  \notag \\
&&+\frac{1}{4}\int_{0}^{1}\left( ta+\left( 1-t\right) b\right) \left\vert
f^{\prime }\left( \left( \frac{1-t}{2}\right) ^{s}a+\left( \frac{1+t}{2}%
\right) ^{s}b\right) \right\vert dt  \notag
\end{eqnarray}%
Since $\left\vert f^{\prime }\right\vert ^{q}\ $is $s$-convex, by the H\"{o}%
lder inequality, we have%
\begin{eqnarray}
&&\int_{0}^{1}\left\vert f^{\prime }\left( \left( \frac{1-t}{2}\right)
^{s}b+\left( \frac{1+t}{2}\right) ^{s}a\right) \right\vert ^{q}dt
\label{3.2} \\
&\leq &\int_{0}^{1}\left( \left( \frac{1-t}{2}\right) ^{s}\left\vert
f^{\prime }\left( b\right) \right\vert ^{q}+\left( \frac{1+t}{2}\right)
^{s}\left\vert f^{\prime }\left( a\right) \right\vert ^{q}\right) dt  \notag
\end{eqnarray}%
and%
\begin{eqnarray}
&&  \label{3.3} \\
&&\frac{1}{4}\int_{0}^{1}\left( tb+\left( 1-t\right) a\right) \left\vert
f^{\prime }\left( \left( \frac{1-t}{2}\right) ^{s}b+\left( \frac{1+t}{2}%
\right) ^{s}a\right) \right\vert dt  \notag \\
&\leq &\frac{1}{4}\left( \int_{0}^{1}\left( tb+\left( 1-t\right) a\right)
^{p}dt\right) ^{\frac{1}{p}}\left( \int_{0}^{1}\left\vert f^{\prime }\left(
\left( \frac{1-t}{2}\right) ^{s}b+\left( \frac{1+t}{2}\right) ^{s}a\right)
\right\vert ^{q}dt\right) ^{\frac{1}{q}}  \notag \\
&&+\frac{1}{4}\left( \int_{0}^{1}\left( ta+\left( 1-t\right) b\right)
^{p}dt\right) ^{\frac{1}{p}}\left( \int_{0}^{1}\left\vert f^{\prime }\left(
\left( \frac{1-t}{2}\right) ^{s}a+\left( \frac{1+t}{2}\right) ^{s}b\right)
\right\vert ^{q}dt\right) ^{\frac{1}{q}}  \notag \\
&\leq &\frac{1}{4}\left( \int_{0}^{1}\left( tb+\left( 1-t\right) a\right)
^{p}dt\right) ^{\frac{1}{p}}\left[ \int_{0}^{1}\left( \left( \frac{1-t}{2}%
\right) ^{s}\left\vert f^{\prime }\left( b\right) \right\vert ^{q}+\left( 
\frac{1+t}{2}\right) ^{s}\left\vert f^{\prime }\left( a\right) \right\vert
^{q}\right) dt\right] ^{\frac{1}{q}}  \notag \\
&&+\frac{1}{4}\left( \int_{0}^{1}\left( ta+\left( 1-t\right) b\right)
^{p}dt\right) ^{\frac{1}{p}}\left[ \int_{0}^{1}\left( \left( \frac{1-t}{2}%
\right) ^{s}\left\vert f^{\prime }\left( a\right) \right\vert ^{q}+\left( 
\frac{1+t}{2}\right) ^{s}\left\vert f^{\prime }\left( b\right) \right\vert
^{q}\right) dt\right] ^{\frac{1}{q}}.  \notag
\end{eqnarray}%
It can be easily seen that%
\begin{equation}
\int_{0}^{1}\left( tb+\left( 1-t\right) a\right) ^{p}dt=\int_{0}^{1}\left(
ta+\left( 1-t\right) b\right) ^{p}dt=\frac{b^{p+1}-a^{p+1}}{\left(
b-a\right) \left( p+1\right) }=L_{p}^{p}\left( a,b\right)  \label{3.4}
\end{equation}%
If expressions (\ref{3.2})-(\ref{3.4}), we obtain%
\begin{eqnarray*}
&&\left\vert \frac{1}{b-a}\int_{a}^{b}f\left( x\right) dx+\frac{af\left(
b\right) -bf\left( a\right) }{2\left( b-a\right) }-\frac{1}{2}f\left( \frac{%
a+b}{2}\right) \right\vert \\
&=&\frac{1}{4}L_{p}\left( a,b\right) \left[ \frac{1}{2^{s}(s+1)}\left(
\left\vert f^{\prime }\left( b\right) \right\vert ^{q}+\left(
2^{s+1}-1\right) \left\vert f^{\prime }\left( a\right) \right\vert
^{q}\right) \right] ^{\frac{1}{q}} \\
&&+\frac{1}{4}L_{p}\left( a,b\right) \left[ \frac{1}{2^{s}(s+1)}\left(
\left\vert f^{\prime }\left( a\right) \right\vert ^{q}+\left(
2^{s+1}-1\right) \left\vert f^{\prime }\left( b\right) \right\vert
^{q}\right) \right] ^{\frac{1}{q}}.
\end{eqnarray*}%
The proof is completed.
\end{proof}

\begin{proposition}
\label{31}Let $a,b\in J^{\circ },\ 0<a<b$ and $s\in \left( 0,1\right] $, then%
\begin{eqnarray}
&&\left\vert L_{s}^{s}\left( a,b\right) +\frac{\left( s-1\right) G^{2}\left(
a,b\right) L_{s-2}^{s-2}\left( a,b\right) -A^{s}\left( a,b\right) }{2}%
\right\vert  \label{se8} \\
&\leq &\frac{L_{p}\left( a,b\right) }{\left( 2^{2q+s}\left( s+1\right)
\right) ^{\frac{1}{q}}}\left\{ \left( \left( sb^{s-1}\right) ^{q}+\left(
2^{s+1}-1\right) \left( sa^{s-1}\right) ^{q}\right) ^{\frac{1}{q}}\right. 
\notag \\
&&+\left. \left( \left( sa^{s-1}\right) ^{q}+\left( 2^{s+1}-1\right) \left(
sb^{s-1}\right) ^{q}\right) ^{\frac{1}{q}}\right\} .  \notag
\end{eqnarray}
\end{proposition}

\begin{proof}
The proof follows from (\ref{se5}) applied\ to the $s$-convex function $%
f\left( x\right) =x^{s}$ and $\left\vert f^{\prime }\left( x\right)
\right\vert =sx^{s-1}.$
\end{proof}

\begin{proposition}
\label{32}Let $a,b\in J^{\circ },\ 0<a<b$ and $s\in \left( 0,1\right) $, then
\end{proposition}

\begin{eqnarray}
&&\left\vert \frac{L_{s}^{s}\left( a,b\right) }{1-s}-\frac{sG^{2}\left(
a,b\right) L_{-1-s}^{-1-s}\left( a,b\right) +A^{1-s}\left( a,b\right) }{%
2\left( 1-s\right) }\ \right\vert  \label{se9} \\
&\leq &\frac{L_{p}\left( a,b\right) }{\left( 2^{2q+s}\left( s+1\right)
\right) ^{\frac{1}{q}}}\left\{ \left( b^{-sq}+\left( 2^{s+1}-1\right)
a^{-sq}\right) ^{\frac{1}{q}}\right.  \notag \\
&&+\left. \left( a^{-sq}+\left( 2^{s+1}-1\right) b^{-sq}\right) ^{\frac{1}{q}%
}\right\} .  \notag
\end{eqnarray}

\begin{proof}
The proof follows from (\ref{se5}) applied\ to the $s$-convex function $%
f\left( x\right) =\frac{x^{1-s}}{1-s}$ and $\left\vert f^{\prime
}(x)\right\vert $ $=1/x^{s}$.
\end{proof}

\begin{remark}
In (\ref{se5}), $\allowbreak $(\ref{se8}), if we take $s\rightarrow 1$, then
(\ref{se5}), $\allowbreak $(\ref{se8}) reduces to (\ref{s2}), \cite[%
Proposition 5]{ms}, respectively.
\end{remark}

\begin{theorem}
\label{4} Let $f:J\rightarrow 
\mathbb{R}
$ be a differentiable function on $J^{\circ }.$ If $\left\vert f^{\prime
}\right\vert $\ is $s$-convex on $\left[ a,b\right] $ for some fixed $s\in
\left( 0,1\right] $ and $q>1$, then%
\begin{eqnarray}
&&  \label{se6} \\
&&\left\vert \frac{1}{b-a}\int_{a}^{b}f\left( x\right) dx-\frac{1}{2}\left( 
\frac{bf\left( a\right) -af\left( b\right) }{b-a}+f\left( \frac{a+b}{2}%
\right) \right) \right\vert  \notag \\
&\leq &\frac{A^{1-\frac{1}{q}}\left( a,b\right) }{\left( 2^{2q+s}\left(
s+1\right) \left( s+2\right) \right) ^{\frac{1}{q}}}\times  \notag \\
&&\left\{ \left[ \left( as+a+b\right) \left\vert f^{\prime }(b)\right\vert
^{q}+\left( b\left( s2^{s+1}+1\right) +a\left( 2^{s+2}-s-3\right) \right)
\left\vert f^{\prime }(a)\right\vert ^{q}\allowbreak \right] ^{\frac{1}{q}%
}\right.  \notag \\
&&+\left. \left[ \left( bs+b+a\right) \left\vert f^{\prime }(a)\right\vert
^{q}+\left( a\left( s2^{s+1}+1\right) +b\left( 2^{s+2}-s-3\right) \right)
\left\vert f^{\prime }(b)\right\vert ^{q}\right] ^{\frac{1}{q}}\right\} 
\notag
\end{eqnarray}
\end{theorem}

\begin{proof}
From\ Lemma\ \ref{l1}\ \ and$\ $using\ the\ well-known power mean\
inequality\ and$\ $since$\ \left\vert f^{\prime }\right\vert ^{q}\ $is$\ s$%
-convex\ on$\ \left[ a,b\right] $, we\ can write%
\begin{eqnarray*}
&&\left\vert \frac{1}{b-a}\int_{a}^{b}f\left( x\right) dx+\frac{af\left(
b\right) -bf\left( a\right) }{2\left( b-a\right) }-\frac{1}{2}f\left( \frac{%
a+b}{2}\right) \right\vert \\
&\leq &\frac{1}{4}\left( \int_{0}^{1}\left( tb+\left( 1-t\right) a\right)
dt\right) ^{1-\frac{1}{q}}\left[ \int_{0}^{1}\left( tb+\left( 1-t\right)
a\right) \left\vert f^{\prime }\left( \frac{1-t}{2}b+\frac{1+t}{2}a\right)
\right\vert ^{q}dt\right] ^{\frac{1}{q}} \\
&&+\frac{1}{4}\left( \int_{0}^{1}\left( ta+\left( 1-t\right) b\right)
dt\right) ^{1-\frac{1}{q}}\left[ \int_{0}^{1}\left( ta+\left( 1-t\right)
b\right) \left\vert f^{\prime }\left( \frac{1-t}{2}a+\frac{1+t}{2}b\right)
\right\vert ^{q}dt\right] ^{\frac{1}{q}} \\
&\leq &\frac{1}{4}\left( \int_{0}^{1}\left( tb+\left( 1-t\right) a\right)
dt\right) ^{1-\frac{1}{q}}\left[ \int_{0}^{1}\left( tb+\left( 1-t\right)
a\right) \left( \left( \frac{1-t}{2}\right) ^{s}\left\vert f^{\prime
}(b)\right\vert ^{q}+\left( \frac{1+t}{2}\right) ^{s}\left\vert f^{\prime
}(a)\right\vert ^{q}\right) dt\right] ^{\frac{1}{q}} \\
&&+\frac{1}{4}\left( \int_{0}^{1}\left( ta+\left( 1-t\right) b\right)
dt\right) ^{1-\frac{1}{q}}\left[ \int_{0}^{1}\left( ta+\left( 1-t\right)
b\right) \left( \left( \frac{1-t}{2}\right) ^{s}\left\vert f^{\prime
}(a)\right\vert ^{q}+\left( \frac{1+t}{2}\right) ^{s}\left\vert f^{\prime
}(b)\right\vert ^{q}\right) dt\right] ^{\frac{1}{q}} \\
&\leq &\frac{1}{4}\left( \frac{a+b}{2}\right) ^{1-\frac{1}{q}}\left[
\left\vert f^{\prime }(b)\right\vert ^{q}\int_{0}^{1}\left( \frac{1-t}{2}%
\right) ^{s}\left( tb+\left( 1-t\right) a\right) dt+\left\vert f^{\prime
}(a)\right\vert ^{q}\int_{0}^{1}\left( \frac{1+t}{2}\right) ^{s}\left(
tb+\left( 1-t\right) a\right) dt\right] ^{\frac{1}{q}} \\
&&+\frac{1}{4}\left( \frac{a+b}{2}\right) ^{1-\frac{1}{q}}\left[ \left\vert
f^{\prime }(a)\right\vert ^{q}\int_{0}^{1}\left( \frac{1-t}{2}\right)
^{s}\left( ta+\left( 1-t\right) b\right) dt+\left\vert f^{\prime
}(b)\right\vert ^{q}\int_{0}^{1}\left( \frac{1+t}{2}\right) ^{s}\left(
ta+\left( 1-t\right) b\right) dt\right] ^{\frac{1}{q}} \\
&\leq &\frac{1}{4}A^{1-\frac{1}{q}}(a,b)\left[ \frac{\left\vert f^{\prime
}(b)\right\vert ^{q}}{2^{s}}\allowbreak \frac{as+a+b}{\left( s+1\right)
\left( s+2\right) }+\frac{\left\vert f^{\prime }(a)\right\vert ^{q}}{2^{s}}%
\frac{b\left( s2^{s+1}+1\right) +a\left( 2^{s+2}-s-3\right) }{\left(
s+1\right) \left( s+2\right) }\right] ^{\frac{1}{q}} \\
&&+\frac{1}{4}A^{1-\frac{1}{q}}(a,b)\left[ \frac{\left\vert f^{\prime
}(a)\right\vert ^{q}}{2^{s}}\frac{bs+b+a}{\left( s+1\right) \left(
s+2\right) }+\frac{\left\vert f^{\prime }(b)\right\vert ^{q}}{2^{s}}\frac{%
a\left( s2^{s+1}+1\right) +b\left( 2^{s+2}-s-3\right) }{\left( s+1\right)
\left( s+2\right) }\right] ^{\frac{1}{q}}.
\end{eqnarray*}%
The proof is completed.
\end{proof}

$\allowbreak $

\begin{proposition}
\label{41} Let $a,b\in J^{\circ }$,$\ 0<a<b$ and $s\in \left( 0,1\right] $,
then\ 
\begin{eqnarray}
&&  \label{se7} \\
&&\left\vert L_{s}^{s}\left( a,b\right) +\frac{\left( s-1\right) G^{2}\left(
a,b\right) L_{s-2}^{s-2}\left( a,b\right) -A^{s}\left( a,b\right) }{2}%
\right\vert  \notag \\
&\leq &\frac{A^{1-\frac{1}{q}}(a,b)}{\left( 2^{2q+s}\left( s+1\right) \left(
s+2\right) \right) ^{\frac{1}{q}}}\times  \notag \\
&&\left\{ \left[ \left( as+a+b\right) \allowbreak \left( sb^{s-1}\right)
^{q}+\left( b\left( s2^{s+1}+1\right) +a\left( 2^{s+2}-s-3\right) \right)
\left( sa^{s-1}\right) ^{q}\right] ^{\frac{1}{q}}\right.  \notag \\
&&+\left. \left[ \left( bs+b+a\right) \left( sa^{s-1}\right) ^{q}+\left(
a\left( s2^{s+1}+1\right) +b\left( 2^{s+2}-s-3\right) \right) \left(
sb^{s-1}\right) ^{q}\right] ^{\frac{1}{q}}\right\} .  \notag
\end{eqnarray}
\end{proposition}

\begin{proof}
The proof follows from (\ref{se6}) applied\ to the $s$-convex function $%
f\left( x\right) =x^{s}$ and $\left\vert f^{\prime }\left( x\right)
\right\vert =sx^{s-1}$.
\end{proof}

\begin{proposition}
\label{42} Let $a,b\in J^{\circ }$,$\ 0<a<b$ and $s\in \left( 0,1\right] $,
then%
\begin{eqnarray}
&&  \label{se10} \\
&&\left\vert \frac{L_{s}^{s}\left( a,b\right) }{1-s}-\frac{sG^{2}\left(
a,b\right) L_{-1-s}^{-1-s}\left( a,b\right) +A^{1-s}\left( a,b\right) }{%
2\left( 1-s\right) }\ \right\vert  \notag \\
&\leq &\frac{A^{1-\frac{1}{q}}(a,b)}{\left( 2^{2q+s}\left( s+1\right) \left(
s+2\right) \right) ^{\frac{1}{q}}}\times  \notag \\
&&\left\{ \left[ \allowbreak \left( as+a+b\right) b^{-sq}+\left( b\left(
s2^{s+1}+1\right) +a\left( 2^{s+2}-s-3\right) \right) a^{-sq}\right] ^{\frac{%
1}{q}}\right.  \notag \\
&&+\left. \left[ \left( bs+b+a\right) a^{-sq}+\left( a\left(
s2^{s+1}+1\right) +b\left( 2^{s+2}-s-3\right) \right) b^{-sq}\right] ^{\frac{%
1}{q}}\right\} .  \notag
\end{eqnarray}
\end{proposition}

\begin{proof}
The proof follows from (\ref{se6}) applied\ to the $s$-convex function $%
f\left( x\right) =\frac{x^{1-s}}{1-s}$ and $\left\vert f^{\prime
}(x)\right\vert $ $=1/x^{s}$.
\end{proof}

\begin{remark}
In (\ref{se6}), $\allowbreak $(\ref{se7}), if we take $s\rightarrow 1$, then
(\ref{se6}), $\allowbreak $(\ref{se7}) reduces to (\ref{s3}), \cite[%
Proposition 8]{ms} respectively.
\end{remark}

\end{document}